\documentclass[11pt,reqno]{amsart}
\usepackage{latexsym}
\usepackage{graphicx}
\usepackage{float}
\usepackage{fancyhdr,amssymb}
\usepackage{mathtools}
\usepackage{tabularx}
\usepackage{comment}

\usepackage{hyperref}

\newcommand{\df}{\dfrac}

\usepackage{amsmath,amsthm}
\renewcommand{\pmod}[1]{\,(\textup{mod}\,#1)}
 \theoremstyle{plain}
\numberwithin{equation}{section}
\newtheorem{theorem}{Theorem}[section]
\newtheorem{proposition}[theorem]{Proposition}
\newtheorem{corollary}[theorem]{Corollary}
\newtheorem{conjecture}[theorem]{Conjecture}
\newtheorem{lemma}[theorem]{Lemma}
\newtheorem{definition}[theorem]{Definition}
\begin{document}
\title[An Arithmetic Sum Associated with the Classical Theta Function]{An Arithmetic Sum Associated \\with the Classical Theta Function}
\author[B.~C.~Berndt, R.~Bhat, J.~Meyer, L.~Xie, A.~Zaharescu]{Bruce C.~ Berndt, Raghavendra Bhat, Jeffrey L.~ Meyer,\\Likun Xie,  Alexandru Zaharescu}
\address{Department of Mathematics, University of Illinois, 1409 West Green
Street, Urbana, IL 61801, USA} \email{berndt@illinois.edu}
\address{Department of Mathematics, University of Illinois, 1409 West Green
Street, Urbana, IL 61801, USA} \email{rnbhat2@illinois.edu}
\address{Department of Mathematics, 215 Carnegie Library, Syracuse University,
Syracuse, NY 13244, USA} \email{jlmeye01@syr.edu}
\address{Max-Planck-Institut f\"{u}r Mathematik, Vivatsgasse 7, 53111, Bonn, Germany} \email{xie@mpim-bonn.mpg.de}
\address{Department of Mathematics, University of Illinois, 1409 West Green
Street, Urbana, IL 61801, USA; Institute of Mathematics of the Romanian
Academy, P.O.~Box 1-764, Bucharest RO-70700, Romania}
\email{zaharesc@illinois.edu}

\begin{abstract}
The sum $S(h,k):=\sum_{j=1}^{k-1}(-1)^{j+1+[hj/k]}$ appears in the modular transformation formulae of the classical theta function $\vartheta_3(z)$. The double sum
$S(k) := \sum_{h=1}^{k-1}S(h,k)$ has a remarkable distribution of values.  Although properties for $S(k)$ and a related sum can be established, several interesting conjectures are open.
\end{abstract}


\maketitle

\section{Introduction and Motivation}\label{section1} The arithmetical sum
\begin{equation}\label{defSkh}
S(h,k) := \sum_{j=1}^{k-1}(-1)^{j+1+[hj/k]},
\end{equation}
where $h$ and $k$ are positive integers and $[x]$ is the greatest integer less than or equal to $x$, appears in the transformation formula for the classical theta function $\vartheta_3(z)$. The sum $S(h,k)$ is analogous to the classical Dedekind sum $s(h,k)$, which appears in the transformation formula for the Dedekind eta-function $\eta(z)$.

For each positive integer $k$, define
\begin{equation}\label{defS}
S(k) := \sum_{h=1}^{k-1}S(h,k).
\end{equation}
Initial calculations led us to make the following conjecture.
\begin{conjecture}\label{mainconjecture} For each odd prime $k>5$,
\begin{equation*}
S(k)>0.
\end{equation*}
\end{conjecture}

Conjecture \ref{mainconjecture} appears to be very difficult to prove.
Further conjectures arising from our calculations are given in Section \ref{conjectures}.

\emph{A priori}, since $S(k)$ is the sum of $(k-1)^2$  summands, each equal to either $+1$ or $-1$, one naturally would expect that neither sign would dominate; i.e., $S(k)$ would be positive roughly half of the time.  But this intuition is apparently misleading.  From the definition of $S(k)$ in \eqref{defS} and our arguments  in the sequel,
we see that the difficulties we face in proving Conjecture \ref{mainconjecture} arise from a more fundamental problem, namely, our lack of knowledge about the distribution of the residues of $hj$ modulo $k$, where $k$ is prime and $1\leq j,h \leq k-1$.

In this paper, we shall also study analogues of $S(h,k)$ and $S(k)$.
To that end,  define
\begin{equation}\label{defT}
T(k) := \sum_{h=1}^{2k-1}T(h,k),
\end{equation}
where
\begin{equation}\label{1.1T}
T(h,k) := \sum_{j=1}^{2k-1}(-1)^{j+1+[hj/k]}.
\end{equation}
In contrast to $S(k)$, we obtain in Section \ref{asymptotic,1} an asymptotic formula for a sum closely related to $T(k)$.

There exists an extensive literature on the Dedekind sum $s(h,k)$. E.g., see the treatise of H.~Rademacher and E.~Grosswald \cite{rg}.  In contrast, little is known about $S(h,k)$. In Section \ref{elementary}, several elementary properties of $S(h,k), S(k), T(h,k)$, and $T(k)$ are proved.

At a more recondite level, the map which sends a pair $(j,k)$ to the congruence class of $hj$ modulo $k$
describes a surface modulo $k$, and one may reinterpret the main questions
discussed in this paper in terms of some delicate/subtle properties of this surface. We
remark in passing that those level curves on this surface given by congruences of the form $hj\equiv C \pmod{k}$,
where $C$ and $k$ are co-prime, are modular hyperbolas modulo $k$. Modular
hyperbolas come with their own interesting distribution problems, results, and conjectures. For a survey on modular
hyperbolas, the reader is referred to I.~E.~Shparlinski's comprehensive paper \cite{shparlinski}.

We close our introduction with an outline of the remainder of the paper.  In the next section, we offer the transformation formula of $\vartheta_3(z)$. Section \ref{conjectures} provides the conjectures, which motivated this paper.   In Section \ref{odd}, we establish some preliminary results, which are needed in Section \ref{asymptotic,1}, in which we prove an asymptotic formula for a sum closely related to $T(k)$. In Section \ref{lowerbounds}, we employ both elementary and analytic methods to establish several lower bounds for $S(k)$, which point out the difficulties we face in proving the conjectures of Section \ref{conjectures}.  Lastly, we establish some elementary facts about $S(k)$ in Section \ref{elementary}.

\section{Transformation Formulas and Background for $S(h,k)$ and $S(k)$}\label{trans}
 Recall that the Dedekind eta-function $\eta(z)$ is defined by
\begin{equation*}
 \eta(z):=q^{1/24}\prod_{n=1}^{\infty}(1-q^n), \quad q=e^{2\pi iz}, \quad z\in \mathbb{H}.
 \end{equation*}
 The Dedekind eta-function satisfies  a well-known transformation formula
\cite[p.~145]{rademacher}.   (Here, and in what follows, $\log w$ denotes the principal value of the logarithm.) If $c$ and $d$ are co-prime integers with $c>0$, and if
$$V(z)=\df{az+b}{cz+d},\quad z \in\mathbb{H},$$
then
\begin{equation*}
\log\{\eta(V(z))\}=\log\{\eta(z)\}+\df{1}{2}\log\left(\df{cz+d}{i}\right)+\pi i s(-d,c),
\end{equation*}
where $s(d,c)$ is  the Dedekind sum defined by
\begin{equation*}
s(d,c):=\sum_{j=1}^{c-1}\left(\left(\df{dj}{c}\right)\right)\left(\left(\df{j}{c}\right)\right),
\end{equation*}
and where
\begin{equation*}
((x)):=\begin{cases} x-[x]-\tfrac12, \quad&\text{ if } x\notin \mathbb{Z},\\
0,&\text{ if } x\in \mathbb{Z}.
\end{cases}
\end{equation*}
Dedekind sums satisfy a famous reciprocity theorem \cite[p.~4, Theorem 1]{rg}.  If $c$ and $d$ are co-prime, positive integers, then
\begin{equation}\label{reciprocity2}
s(c,d)+s(d,c)=-\df14 +\df{1}{12}\left(\df{c}{d}+\df{1}{cd}+\df{d}{c}\right).
\end{equation}

 Next, recall that the classical theta function $\vartheta_3(z)$ is defined by
\begin{equation*}
\vartheta_3(z):=\sum_{n=-\infty}^{\infty}q^{n^2}, \quad q=e^{\pi iz}, \quad z\in \mathbb{H}.
\end{equation*}
The function  $\log\{\vartheta_3(z)\}$ obeys a modular transformation formula  \cite[p.~339, Theorem 4.1]{berndt-eisenstein}.
 If $c$ and $d$ are co-prime integers of opposite parity with $c>0$, and if
$$V(z)=\df{az+b}{cz+d}, \quad z\in\mathbb{H},$$
then
\begin{equation*}
\log\{\vartheta_3(V(z)\}=\log\{\vartheta_3(z)\}+\frac12\log(cz+d)-\frac14 \pi i + \frac14 \pi i S(d,c),
\end{equation*}
where $S(d,c)$ is defined by \eqref{defSkh}.
The sum $S(d,c)$ satisfies a beautiful reciprocity theorem \cite[p.~339, Theorem 4.2]{berndt-eisenstein}.  Let $c$ and $d$ be positive, co-prime integers of opposite parity.  Then,
\begin{equation}\label{reciprocity}
S(d,c)+S(c,d)=1.
\end{equation}
(Note that the sum $S(d,c)$ has one less term than the corresponding sum appearing in \cite[p.~339]{berndt-eisenstein}.)
 Thus,  $S(d,c)$ is an analogue of the Dedekind sum $s(d,c)$, and   the reciprocity theorem \eqref{reciprocity} is an analogue of \eqref{reciprocity2}.

\section{Numerical Calculations and Conjectures}\label{conjectures}

The tables that follow give the first few values for $S(k)$ and $T(k)$.  In contrast to $S(k)$, all of the values of $T(k)$ are negative. See Lemma \ref{classify}, which verifies this observation.

\begin{table}[!htb]
  \caption{$S(k)$ values}
  \centering
  \small 
  \setlength{\tabcolsep}{3pt} 
  \begin{tabular}{|c|c|c|c|c|c|c|c|c|c|c|c|c|c|c|c|c|c|c|c|c|}
    \hline
    \textbf{$k$}& 1& 2& 3& 4& 5& 6& 7& 8& 9& 10& 11& 12& 13& 14& 15& 16& 17& 18& 19& 20\\
    \hline
    \textbf{$S(k)$} &0 & 1& 2& 5& 4& 7& 10& 11& 11& 8& 17& 14& 21& 20& 15& 18& 39& 24& 21& 38\\
    \hline
  \end{tabular}
\end{table}

\begin{table}[!htb]
  \caption{$T(k)$ values}
  \centering
  \small 
  \setlength{\tabcolsep}{3pt} 
  \begin{tabular}{|c|c|c|c|c|c|c|c|c|c|c|c|c|c|c|c|c|c|c|c|c|}
    \hline
    \textbf{$k$}& 1& 2& 3& 4& 5& 6& 7& 8& 9& 10& 11& 12& 13& 14& 15& 16& 17& 18& 19& 20\\
    \hline
    \textbf{$T(k)$} &-1 & -1& -5& -1& -9& -9& -13& -1& -25& -17& -21& -17& -25& -25& -61& -1& -33& -49& -37& -33\\
    \hline
  \end{tabular}
\end{table}

Numerical calculations suggest that Conjecture \ref{mainconjecture} can be strengthened. First, there are only 17 primes less than 10000, namely,     2, 3, 5, 7, 11, 13, 17, 23, 29, 41, 53, 59, 83, 113, 149, 179, 233,  for which $S(k)$ is  less than $2k$.  Thus, we record the following conjecture.

\begin{conjecture}\label{conjecturemain3}
For each prime $k> 233$,
\begin{equation*}
    S(k)>2k.
    \end{equation*}
    \end{conjecture}

 Moreover, with 3119 being the largest, there are only 87 primes less than 10000 that fail the inequality  $S(k) < 3k$:
    2, 3, 5, 7, 11, 13, 17, 19, 23, 29, 31, 37, 41, 43, 47, 53, 59, 61, 67, 71, 73, 79, 83, 89, 97, 101, 103, 107, 109, 113, 131, 137, 139, 149, 163, 167, 173, 179, 193, 197, 233, 239, 251, 257, 263, 269, 293, 317, 347, 349, 359, 383, 389, 419, 439, 443, 449, 479, 503, 509, 557, 563, 569, 593, 599, 683, 719, 743, 797, 809, 827, 839, 863, 1013, 1019, 1049, 1103, 1229, 1259, 1409, 1733, 1889, 1913, 2339, 2459, 2969, 3119.  Accordingly, we make the following conjecture:

    \begin{conjecture}\label{conjecturemain4}
For each prime $k>3119$,
\begin{equation*}
    S(k)>3k.
    \end{equation*}
    \end{conjecture}

    For primes $k \leq 50000$, there are 8 pairs
    $$\{(32603, 126466), (33149, 126068), (34649, 134104), (34913, 137712)\},$$
     $$\{(35573, 137420), (41579, 165026), (44909, 175916), (49139, 189522)\}$$
      that do not satisfy the inequality $S(k)>4k$ when $k$ is a prime. In particular, note that the largest value of $k$ in this set is 49139, which is close to 50000. Therefore, on the basis of our data, at this  point, we cannot make the conjecture $S(k)>4k$ for primes $k$.

As we increase the number of calculated values of $S(k)$, where $k$ is prime, the slope of the best fit straight line increases.     The best fit straight line and the best fit (quadratic) curve (both colored black on the graph) are extremely close together, indicating an almost linear growth, as claimed  in Conjecture \ref{C1} below. The best fit straight line has a slope of approximately 5.7.

On the other hand, if we attempt to find best fit curves involving higher powers of $k$ and logarithms, we do not find any fits that are convincingly more accurate.  Thus, $S(k)$ appears to have a growth rate that is slightly greater than linear.  With some temerity, we then make the following conjecture.

\begin{conjecture}\label{C1} Let $C_1$ and $C_2$ be any fixed positive numbers. Let $\epsilon >0$ be fixed. Then, for all primes $k$ sufficiently large,
$$ C_1 k < S(k) <C_2 k^{1+\epsilon}.$$
\end{conjecture}

Although the summands of $S(k)$ are only trivial exponential sums, the second inequality in Conjecture \ref{C1} is perhaps analogous to conjectures in prime number theory for which   ``square-root cancellation'' is surmised.

We provide two graphs below.  The first image depicts the values of $S(k)$ for $k=1$ to $k=10000$. Note that there are several negative values of $S(k)$. In the second graph, for primes $k$, such that $ 1\leq k \leq 50000$, the values of  $S(k)$ are depicted.

\begin{figure}[H]
\includegraphics[width=10cm]{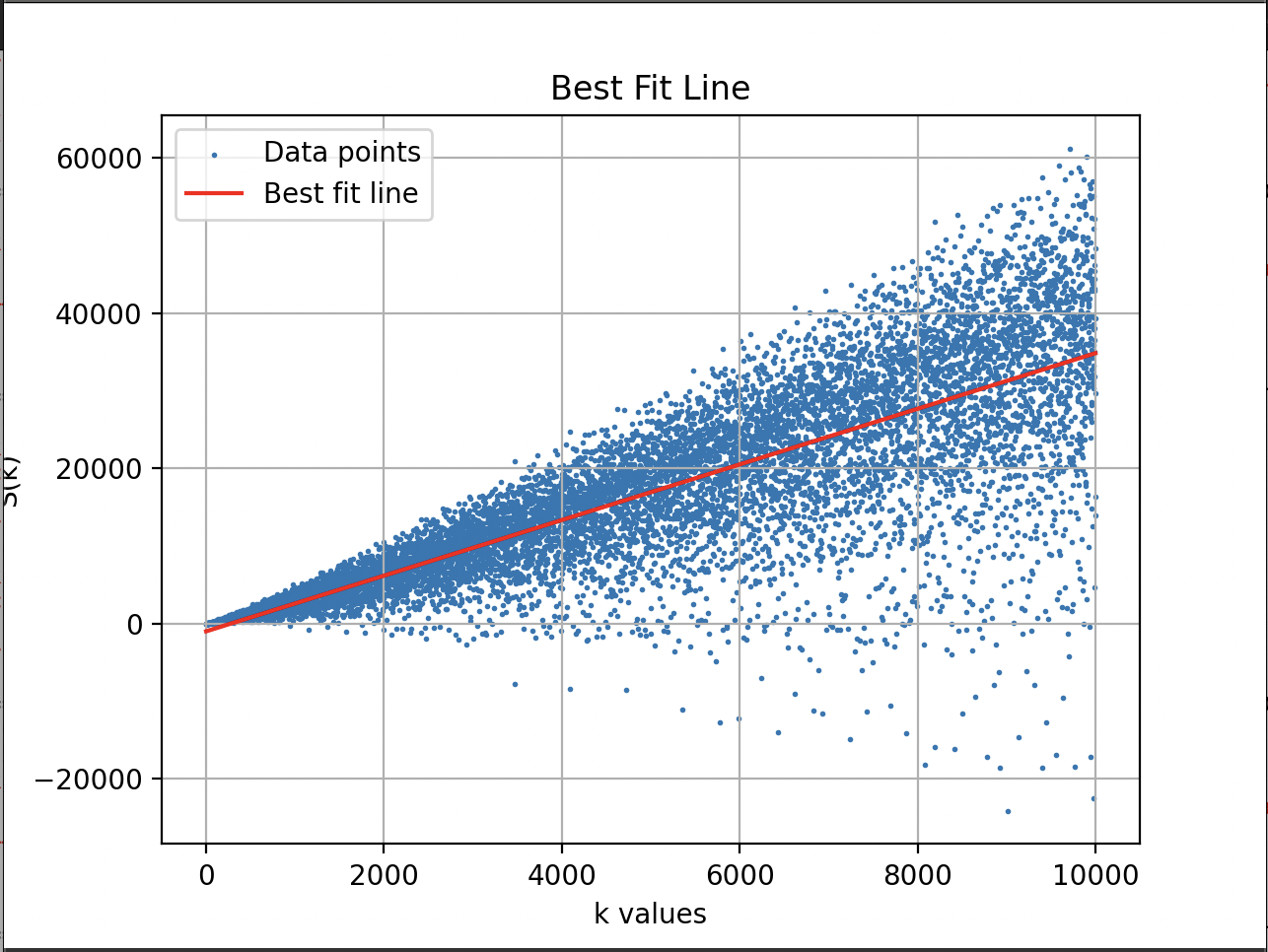}
\end{figure}

\begin{figure}[H]
\includegraphics[width=12cm]{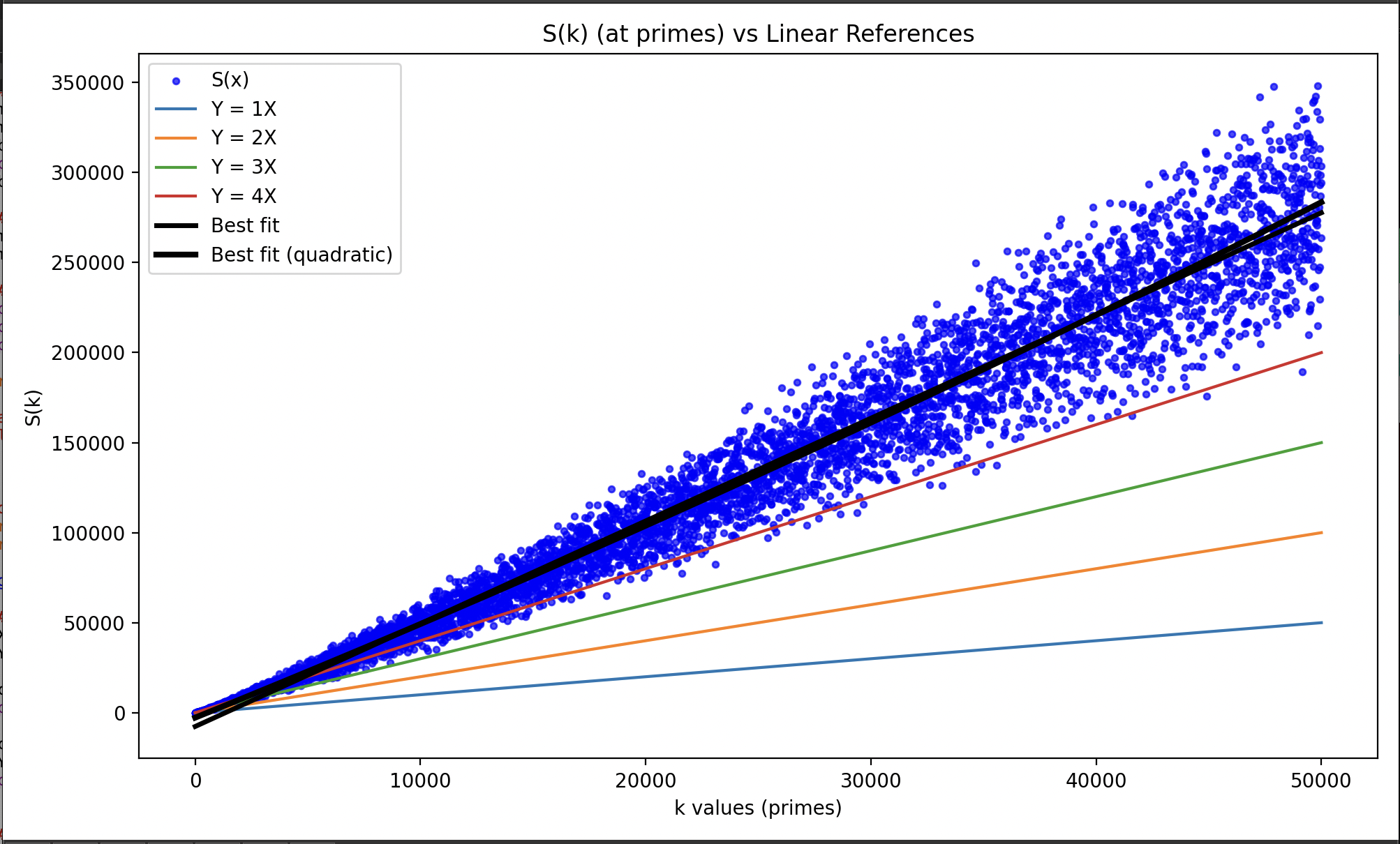}
\end{figure}

We offer some observations about the negative values of $S(k)$.  As seen from our graph, we have calculated $S(k)$ for \emph{all} $k$ up to $10000$.

For $k\leq 10000$, there are  39 negative values of $S(k)$. For some of these,  $k$ is a multiple of 3, but not a multiple of 5, with the first being $S(2079)=-1390$ and the last being $S(9933)=-448$.

For $k\leq 10000$, there are 8 negative values of $S(k)$ when $k$ is a multiple of 5, but not a multiple of 3, with the first being $S(5005)=-1332$ and the last being $S(8855)=-7950$.

For $k\leq 10000$, there are 104 negative values of $S(k)$ when $k$ is a multiple of \emph{both} 5 and 3, with the first being $S(945)=-296$ and the last being $S(9975)=-22450$.

Hence, altogether, there are 151 negative values of $S(k)$ for $k\leq 10000$.  Contrary to what one might expect, for negative values of $S(k)$, there are many more values of $k$ divisible by 15 than there are divisible by 3 or 5 only.

In a written communication, Wolfgang Berndt observed that $k$ need not be a multiple of 3 or 5 for $S(k)$ to be negative.  For example, $S(17017)=S(7\cdot11\cdot13\cdot17)=-2364$, and $S(19019)=S(7\cdot11\cdot13\cdot19)=-20578$.
Furthermore,  $S(k)$ can even be smaller than $-k$. For example, $S(3465)=  -7800$ and   $S(10395)=  -40726$. Thus,  despite the positive bias of $S(k)/k$ to grow stronger for primes $k$,  $S(k)/k$  is perhaps also unlimited on the negative side, especially for ``very composite'' $k$.

\section{Preliminary Results on $T(k)$}\label{odd}

In Section \ref{asymptotic,1},  Theorem \ref{important} gives an asymptotic formula for a sum closely related to $T(k)$, defined in \eqref{defT}.  In this section, we lay the groundwork for Theorem \ref{important} by establishing Lemmas \ref{classify} and \ref{classify2} below. In particular, we can conveniently express $T(k)$ by  a single summation, in contrast to $S(k)$, for which we have representations by only  double sums. Moreover, our formula does not involve sums over powers of $-1$, but instead sums involving the Euler's totient function $\phi(n)$, defined by
$$ \phi(n)= \left\{m: 1\leq m < n; \gcd(m,n)=1\right\}. $$

Throughout this section, $k$ is an odd positive integer greater than 1.

\begin{lemma}\label{classify} If $k$ is an odd positive integer, then
\begin{equation*}
T(k)= 2k - 1 - 2\sum\limits_{d|k} d\cdot\phi\left(\frac{k}{d}\right).\label{classify1}
\end{equation*}
\end{lemma}

We show that Lemma \ref{classify} can be recast using  the following lemma.

\begin{lemma}\label{classify2} For odd $k$,
\begin{equation*}\label{classify3}
\sum\limits_{d|k} d\cdot\phi\left(\frac{k}{d}\right)=\sum\limits_{\substack{1\leq h \leq 2k - 1 \\ h \textup{ odd}}}  \gcd(k,h).
\end{equation*}
\end{lemma}

\begin{proof}
\allowdisplaybreaks
We have
\begin{align}\label{classify4}
\sum\limits_{\substack{1\leq h \leq 2k - 1 \\ h \text{ odd}}}  \gcd(h,k) &= \sum\limits_{d|k}d\cdot \# \{h:\gcd(h,k) = d, 1\leq h \leq 2k-1, h \text{ odd }\}\notag\\
&=\sum\limits_{d|k} d\cdot\# \left\{h:\gcd\left(\frac{h}{d},\frac{k}{d}\right) = 1, 1\leq h \leq 2k-1, h \text{ odd }\right\}\notag\\
&=\sum\limits_{d|k} d\cdot \# \left\{m:\gcd\left(m,\frac{k}{d}\right) = 1, 1\leq m\leq \left[\dfrac{2k-1}{d}\right], m \text{ odd }\right\}\notag\\
&=\sum\limits_{d|k} d\cdot\phi\left(\frac{k}{d}\right),
\end{align}
since $d\nmid j$, for $k+1\leq j\leq 2k-1.$
\end{proof}

We prove Lemma \ref{classify} with the help of a few useful lemmas.

\begin{lemma}\label{g2} For $h$ and $k$ both odd,
$$T(h,k) = \sum\limits_{\substack{1\leq j \leq 2k - 1 \\ hj/k \in \mathbb{Z}}} (-1).$$
\end{lemma}

\begin{proof}
For brevity, set
$$ T(h,j,k):=(-1)^{j+1+[hj/k]}.$$
By elementary considerations,
\begin{equation*}
T(h,2k-j,k)=\begin{cases} -T(h,j,k),\quad &\text{if } \frac{hj}{k} \notin \mathbb{Z},\\
T(h,j,k),\quad &\text{if } \frac{hj}{k} \in \mathbb{Z}.
\end{cases}
\end{equation*}
Therefore, we conclude that
$$T(h,k)=\sum\limits_{1\leq j \leq 2k-1}T(h,j,k) = \sum\limits_{\substack{1\leq j \leq 2k - 1 \\ hj/k\in \mathbb{Z}}}(-1)^{j+1+hj/k}.$$
Since the parity of $\frac{hj}{k}$ matches the parity of $j$ when $\frac{hj}{k}\in\mathbb{Z}$, and since $h$ and $k$ are both odd, we see that
$(-1)^{j+1+hj/k} = -1.$
Thus,
$$T(h,k) = \sum\limits_{\substack{1\leq j \leq 2k - 1 \\ hj/k \in \mathbb{Z}}} (-1).$$
\end{proof}

\begin{lemma}\label{g3} For $h$ and $k$ both odd,
$$T(h,k) = 1 - 2\cdot \gcd(h,k). $$
\end{lemma}
\begin{proof}
Let $d=\gcd(h,k)$. The least common multiple of $h$ and $k$ can be expressed as
$${\rm lcm}(h,k)=\frac{h k}{d}.$$
Thus, for some $n \in \mathbb{N}$,
$$\frac{hj}{k}\in\mathbb{Z}\quad \Leftrightarrow \quad j=\frac{n k}{d}.$$
Since $1\leq j\leq 2k-1,$ we have $1\leq n< 2d.$ Thus, by Lemma \ref{g2},
$$T(h,k)=\sum\limits_{1\leq n<2d}T\left(h,\frac{nk}{d},k\right) = \sum\limits_{1\leq n<2d} (-1) = -(2d-1) = 1- 2\cdot \gcd(k,h).$$
\end{proof}

\begin{proof}[Proof of Lemma \ref{classify}.] From Corollary \ref{g1} (to be proved in Section \ref{elementary}) and Lemma \ref{g3},
\begin{align*}
T(k) =&\sum\limits_{\substack{1\leq h \leq 2k - 1 \\ h \text{ even}}}  1 + \sum\limits_{\substack{1\leq h \leq 2k - 1 \\ h \text{ odd}}}(1 -2\cdot \gcd(h,k))\\
=&\,\, 2k-1 - 2\sum\limits_{\substack{1\leq h \leq 2k - 1 \\ h \text{ odd}}}  \gcd(h,k).
\end{align*}
Lemma \ref{classify} now follows from Lemma \ref{classify2}.
\end{proof}

\section{Asymptotic Formula for Partial Sums Related to $T(k)$}\label{asymptotic,1}
Recall from Lemma \ref{classify} that
\begin{equation}
T(k)= 2k - 1 - 2\sum\limits_{d|k} d\cdot\phi\left(\frac{k}{d}\right).\label{classify77}
\end{equation}
Motivated by \eqref{classify77}, we derive an asymptotic formula for the sum on the right-hand side above.
Set
\begin{equation}\label{f4}
a_n := \begin{cases}
\displaystyle{\sum_{d|n}d\cdot\phi\left(\frac{n}{d}\right)}, \quad &\text{ if }  n  \text{ is odd},\\
0, &\text{ if } n \text{ is even}.
\end{cases}
\end{equation}
We define a generating function for $a_n$ by
\begin{equation}\label{f1}
F(s) := \sum_{n=1}^{\infty} \frac{a_n}{n^s}, \qquad \Re( s )>2.
\end{equation}
Next, define
\begin{equation}\label{f2}
G(s) := \sum_{n = 1}^{\infty}\frac{\sum_{d|n}d\cdot \phi\left(\frac{n}{d}\right)}{n^s} = \frac{\zeta(s-1)^2}{\zeta(s)} = \prod_{p \text{ prime}}\frac{1-\frac{1}{p^s}}{\left(1-\frac{1}{p^{s-1}}\right)^2},
\end{equation}
where the second identity can be established by multiplying the quotient of zeta functions on the right-hand side.  (See also \cite[p.~6, Equation (1.2.12)]{zeta}.)
It follows from \eqref{f4}--\eqref{f2} that
\begin{equation}\label{f3}
F(s) =  \prod_{p \text{ odd prime}}\frac{1-\frac{1}{p^s}}{\left(1-\frac{1}{p^{s-1}}\right)^2} = \frac{\left(1-\frac{1}{2^{s-1}}\right)^2}{1-\frac{1}{2^s}}\cdot G(s) = \frac{\left(1-\frac{1}{2^{s-1}}\right)^2}{1-\frac{1}{2^s}}\cdot \frac{\zeta(s-1)^2}{\zeta(s)}.
\end{equation}


We come to the most important result in this section.

\begin{theorem}\label{important}
   For each $\epsilon$ such that $0.134<\epsilon<\frac{1}{2}$, we have
    \begin{equation*}\label{f6}
        {\sum_{n\leq x}}^{\prime}a_n  =\dfrac{x^2}{\pi^2}\log x+\df{1}{6\pi^2}\left(-36\zeta^{\prime}(2)+A\right)x^2+O_{\epsilon}\left(x^{\frac{3}{2}+\epsilon}\log^2 x\right),
        \end{equation*}
        where
        \begin{equation*}
        A=12\gamma-3+10\log 2,
        \end{equation*}
        $\gamma$ denotes Euler's constant,
    and the prime $'$ on the summation sign indicates that if $x$ is an integer, then we count only $\frac{1}{2}a_x$.
\end{theorem}

\begin{proof} By Perron's summation formula \cite[pp.~12--14]{hardy-riesz},
\begin{equation}\label{perron1}
{\sum_{n\leq x}}^{\prime}a_n = \frac{1}{2\pi i}\int_{c-i\infty}^{c+i\infty}F(s)\frac{x^s}{s} ds = \frac{1}{2\pi i}\int_{c-iT}^{c+iT}F(s)\frac{x^s}{s} ds + E(T),
\end{equation}
where $F(s)$ is defined in \eqref{f3} and
\begin{equation}\label{E}
E(T)=\frac{1}{2\pi i}\left(\int_{c-i\infty}^{c-iT}+\int_{c+iT}^{c+i\infty}\right)F(s)\frac{x^s}{s} ds.
\end{equation}

 Let $c>2$ and $\frac32\leq \sigma_0 <c$, where $c$ and $\sigma_0$ will be determined later.  Consider
\begin{equation*}
    I_T:= \frac{1}{2\pi i}\int_{C_T}F(s)\frac{x^s}{s} ds=
    \frac{1}{2\pi i}\int_{C_T}\frac{\left(1-\frac{1}{2^{s-1}}\right)^2}{1-\frac{1}{2^s}}\cdot \frac{\zeta(s-1)^2}{\zeta(s)}\frac{x^s}{s} ds,\label{integral}
\end{equation*}
where $C_T$ is the positively oriented rectangle with vertices $\sigma_0 \pm iT$ and $c\pm iT$.
On the interior of $C_T$, the integrand is analytic except for a double pole at $s=2$, arising from $\zeta^2(s-1)$.  Let $R_a\bigl(f(s)\bigr)$  denote the residue of a function $f(s)$ at $s=a$.   Therefore, by the residue theorem,
\begin{align}\label{residuetheorem}
I_T=R_2:=&R_2\left(\frac{\left(1-\frac{1}{2^{s-1}}\right)^2}{1-\frac{1}{2^s}}\cdot \frac{\zeta(s-1)^2}{\zeta(s)}\frac{x^s}{s}\right)\notag\\
    =& \frac{-36x^2\zeta'(2)+12\gamma \pi^2x^2-3\pi^2x^2+10\pi^2x^2\mathrm{log}(2)+6\pi^2x^2\mathrm{log}(x)}{6\pi^4},
\end{align}
where the residue can be calculated either by hand or by \emph{Mathematica}.

We next calculate $I_T$ by a second method.  Let $I_1,I_2,I_3,I_4$ denote, respectively, the integrals over
\begin{align*}
    [c-iT,c+iT],\quad [c+iT,\sigma_0 +iT],\quad
[\sigma_0+iT,\sigma_0-iT],\quad
[\sigma_0-iT,c-iT].
\end{align*}

We first find a suitable bound for $I_2$.  Accordingly,
\begin{align}
    I_2&=-\frac{1}{2\pi i}\int_{\sigma_0+iT}^{c+iT}\frac{\left(1-\frac{1}{2^{s-1}}\right)^2}{1-\frac{1}{2^s}}\cdot \frac{\zeta(s-1)^2}{\zeta(s)}\frac{x^s}{s} ds\notag\\&
    \ll \frac{1}{T}\int_{\sigma_0}^cx^u|\zeta(u-1+iT)|^2\,du.\label{I_3}
\end{align}
From \cite[p.~96, Equation (5.1.4)]{zeta}, we
have
\[|\zeta(u-1+iT)|\ll T^{1-(\sigma_0-1)},\]
when $\sigma_0\leq u<2$, and
\[|\zeta(u-1+iT)|\ll \log T, \]
 when $2\leq u\leq 3$.
 Using these bounds in \eqref{I_3}, we obtain
\begin{equation}
    I_2\ll  \frac{1}{T^{2\sigma_0-3}} \int_{\sigma_0}^2 x^u \,du+ \frac{\log^2 T}{T}\int_{2}^c x^u \,du\ll  \frac{x^2}{T^{2\sigma_0-3}\log x}+\frac{\log^2 T\,\, x^c}{T\log x}
    \label{I33}.
\end{equation}
By the same argument, this bound holds for $I_4$ as well.

We now turn to bounding $I_3$.  To that end,
 \begin{align}
    I_3&=\frac{1}{2\pi i} \int_{\sigma_0-iT} ^{\sigma_0+iT} \frac{\left(1-\frac{1}{2^{s-1}}\right)^2}{1-\frac{1}{2^s}}\cdot \frac{\zeta(s-1)^2}{\zeta(s)}\frac{x^s}{s}\, ds \nonumber\\
    &\ll x^{\sigma_0}\int_{-T}^{T} \frac{|\zeta(\sigma_0-1+it)|^2}{|\sigma_0+it|}\,dt\nonumber\\
   & \ll x^{\sigma_0}+x^{\sigma_0}\int_1^T\frac{|\zeta(\sigma_0-1+it)|^2}{t}\,dt. \label{I_2}
\end{align}
Using a dyadic decomposition, we find, from \eqref{I_2}, that
\begin{equation}\label{I3}
   I_3 \ll x^{\sigma_0}+x^{\sigma_0}  \sum_{v=0}^{[ \log_2 T]}\frac{1}{2^v} \int _{2^v} ^{2^{v+1} } |\zeta(\sigma_0-1 +it)|^2\, dt.
\end{equation}
Employing  \cite[p.~118, Theorem 7.2(A)]{zeta}, we find that 
\begin{equation*}
    \int_{2^v} ^{2^{v+1}} |\zeta(\sigma_0-1 +it)|^2\,dt \ll v2^v.
\end{equation*}
Using the bound above in \eqref{I3}, we arrive at 
\begin{align}
    I_3&\ll x^{\sigma_0}+x^{\sigma_0} \sum_{v=0} ^{[ \log_2 T ]} v \ll  x^{\sigma_0}\log^2 T. \label{I_2_estimate}
\end{align}

In summary, by  \eqref{I33} and \eqref{I_2_estimate}, we deduce that
\begin{equation}\label{III}
 I_2+I_3+I_4 \ll   \frac{x^2}{T^{2\sigma_0-3}\log x}+\frac{\log^2 T\,\, x^c}{T\log x}+ x^{\sigma_0}\log^2 T.
 \end{equation}

To optimize the bounds for $I_2, I_3$, and $I_4$, we seek suitable constants such that the maximal exponents of $x$ in the estimates of $I_2$ and $I_4$ are  equal.  We also desire the maximal exponent of $x$ in the estimate of $I_3$ to be very close to the exponents in the estimates of  $I_2$ and $I_4$.

 Let
 \begin{equation}\label{maxmax}
  \sigma_0=\frac{3}{2}+\epsilon,
  \end{equation}
  for $0.134<\epsilon <\frac{1}{2}$.
If we take
\begin{equation}\label{G2}
c= 3 \quad\text{ and } \quad T=x^{a},
\end{equation}
 where $a$ is to be determined,
 the maximal exponent of $x$ in the estimate of $I_2$ and $I_4$ is
 $$\mathrm{max}\{2-a(2\sigma_0-3), c-a\}.$$
 To make them equal, we take $$a=\frac{1}{1-2\epsilon}.$$
 Then,
  \begin{equation}\label{G2_1}
  \mathrm{max}\{2-a(2\sigma_0-3), c-a\}=3-\frac{1}{1-2\epsilon}.
  \end{equation}

  Compare the exponent above with the exponent from the bound of $I_3$, i.e., $\frac{3}{2}+\epsilon$, by \eqref{maxmax}. For $\epsilon>0.134$,  we find that
$$3-\frac{1}{1-2\epsilon} < \frac{3}{2}+\epsilon.$$
   Putting these together in \eqref{III}, we conclude that
   \begin{equation}\label{IIIa}
  I_2+I_3+I_4 =  O_\epsilon(x^{\frac{3}{2}+{\epsilon}}\log^2x).
  \end{equation}
  Combining \eqref{IIIa} with \eqref{residuetheorem}, we further conclude that
  \begin{equation}\label{IV}
  I_1=I_T-I_2-I_3-I_4 = R_2+ O_\epsilon(x^{\frac{3}{2}+{\epsilon}}\log^2x).
  \end{equation}

  Returning to \eqref{perron1}, we see that it remains to bound $E(T)$, which is given by \eqref{E}.  With $d(n)$ denoting the number of positive divisors of $n$ below, we find that
\begin{align}
        E(T)\leq &\frac{x^c}{T}\sum _{n=1}^\infty \frac{\sum_{d\mid n }d\phi(\frac{n}{d})}{n^c|\log(\frac{x}{n})|}=\frac{x^c}{T}\sum _{n=1}^\infty \frac{\sum_{d\mid n }\frac{n}{d}\phi(d)}{n^c |\log(\frac{x}{n})|}
        =\frac{x^c}{T}\sum _{n=1}^\infty \frac{\sum_{d\mid n }\frac{\phi(d)}{d}}{n^{c-1} |\log(\frac{x}{n})|}\nonumber\\
        \leq&  \frac{x^c}{T}\sum _{n=1}^\infty \frac{d(n)}{n^{c-1} |\log(\frac{x}{n})|}
        \leq 2\frac{x^c}{T} \sum_{n=1 }^\infty \frac{1}{n^{c-3/2}|\log (\frac{x}{n})|},\label{E(T)}
           \end{align}
 where we used the easy bound $d(n)\leq 2n^{1/2}$.

 We separate two cases.  First, let $E_1(T)$ denote those summands on the right side of \eqref{E(T)}  for which  $n \leq \frac{3}{4} x$ or $n\geq \frac{5}{4} x$. For these terms,  $|\log(\frac{x}{n})|$ has a positive lower bound, and hence,
  \begin{equation}\label{E1}
 E_1(T) \ll \frac{x^c}{T}\df{1}{x^{c-5/2}}=\df{x^{5/2}}{T}.
 \end{equation}

  Second, let $E_2(T)$ denote those summands in $E(T)$ for which $\frac{3x}{4}<n<\frac{5x}{4}.$
 Observe that  $\sum_{n\leq y}a_n $
 does not change in each interval, $N < y < N + 1$, for any fixed positive
integer $N$. Therefore, we may choose $x = N_0 + 1/2$, for each positive integer $N_0$. Then for
 $\frac{3}{4}x <n<\frac{5}{4} x$, by the mean value theorem,
\begin{equation*}\label{EE}
\left|\log \left(\frac{x}{n}\right)\right|\geq \frac{|x-n|}{x}\geq \frac{1}{2x}.
\end{equation*}
Therefore, the contribution in this case is
\begin{equation}\label{E2}
E_2(T)\ll \frac{x^{c+1} }{T} \sum_{\frac{3}{4}x <n<\frac{5}{4} x}\frac{1}{n^{c-3/2}}\ll \frac{x^{7/2} }{T}.
\end{equation}
 Combining \eqref{E1} and \eqref{E2} and also using \eqref{III} and \eqref{G2},  we deduce that
 \begin{equation}\label{Efinal}
 E(T)\ll x^{7/2-a}\ll x^{3/2}.
 \end{equation}

 In summary, from \eqref{perron1}, \eqref{IV}, and \eqref{Efinal}, we have shown that
 \begin{align*}
 {\sum_{n\leq x}}^{\prime}a_n=&I_1+E(T)\\
 =&R_2+ O_\epsilon(x^{\frac{3}{2}+{\epsilon}}\log^2x)+O(x^{3/2})\\
 =&R_2+ O_\epsilon(x^{\frac{3}{2}+{\epsilon}}\log^2x),
 \end{align*}
 where $R_2$ is given by \eqref{residuetheorem}.  This completes the proof of Theorem \ref{important}.
  \end{proof}

\section{Lower bounds for  $S(k)$}\label{lowerbounds}

In considering Conjecture \ref{mainconjecture}, we seek lower bounds for  $S(k)$.  We propose several such bounds in this section. Our first bounds are simple.

From Proposition \ref{4.7}, $S(h,k)=0$, when $h$ and $k$ are co-prime positive integers.  Thus, we can write

\begin{align}
	S(k)=\sum_{\substack{j=1\\j\text{ odd}}}^{k-1}\sum_{h=1}^{k-1} (-1)^{j+1+[hj/k]}
	=\sum_{\substack{j=1\\j\text{ odd}}}^{k-1}\sum_{h=1}^{k-1} (-1)^{[hj/k]}
	=\frac{(k-1)^2}{2}-2\sum_{\substack{j=1\\j\text{ odd}}}^{k-1}r(j,k), \label{S_k_expression}
\end{align}
where
\begin{equation}\label{r}
r(j,k)= \#\left\{h: 1\leq h\leq k-1;  \left[\df{hj}{k}\right] \text{ is odd }\right\}.
\end{equation}
Observe that if we assume that each value of $\left[\frac{hj}{k}\right]$ is odd, or that each value  of $\left[\frac{hj}{k}\right]$ is even, then we obtain the trivial bounds
\begin{equation}\label{trivial}
-\df{(k-1)^2}{2}\leq S(k) \leq \df{(k-1)^2}{2}.
\end{equation}
Using the elementary relation,
\begin{equation}
    \label{squarebrackets}
[x]-2\left[\dfrac{x}{2}\right] = \begin{cases} 0, \quad &\text{ if $[x]$ is even},\\
1, &\text{ if $[x]$ is odd},
\end{cases}
\end{equation}
and the elementary evaluation (see \cite[p.~32, Equation (41)]{rg} or \cite[p.~214, Equation (4.6)]{berndt-dieter}),
 \begin{equation}\label{elementary evaluation}
\sum_{j=1}^{n-1} \left[\frac{mj}{n}\right] = \frac{(m-1)(n-1)}{2},
\end{equation}
where $m$ and $n$ are co-prime,
we can write
\begin{align}
    r(j,k)= \sum_{h=1}^{k-1} \left(\left[\df{hj}{k}\right]-2\left[\df{hj}{2k}\right]\right)
    =\frac{(j-1)(k-1)}{2}-2\sum_{h=1}^{k-1}\left[\df{hj}{2k}\right].\label{m_jk_expression}
  \end{align}
  (The sum over $h$  on the right-hand side  does  not have a known closed form.)
If we write $j=2b+1$, then
$$\left[\frac{hb}{k}\right]\leq \left[\frac{hj}{2k}\right]=\left[\frac{h(2b+1)}{2k}\right]\leq \left[\frac{h(b+1)}{k}\right].$$
Using these trivial bounds for the floor function in \eqref{m_jk_expression}, and then using the resulting inequalities in \eqref{S_k_expression}, we obtain the same trivial bounds  \eqref{trivial}.

\begin{theorem}\label{4.9theorem} If $k$ is an odd prime, then
\begin{align*}
S(k)
=&\; -(k-1)^2 + 4 \sum_{\ell=1}^{\frac{k-1}{2}}\sum_{h=1}^{\frac{k-1}{2}}\left(\left\{\frac{2h\ell}{k}\right\} + \left\{\frac{h(2\ell-1)}{k}-\frac{1}{2}\right\}\right).
\end{align*}
\end{theorem}

\begin{proof} By Proposition \ref{4.7}, $S(h,k) = 0$ when $h$ and $k$ are both odd; so we can write
\begin{align}\label{4.10}
S(k)=&\; \sum_{h=1}^{\frac{k-1}{2}}S(2h,k) = \sum_{h=1}^{\frac{k-1}{2}}\sum_{j=1}^{k-1}(-1)^{j+1+\left[2hj/k\right]}\notag\\
=&\;-\sum_{j=1}^{k-1}\sum_{h=1}^{\frac{k-1}{2}} (-1)^{j+\left[2hj/k\right]}\notag\\
=&\;-\sum_{j=1}^{k-1}\sum_{h=1}^{\frac{k-1}{2}} (-1)^{\left[(2hj+jk)/k\right]}\notag\\
=&\; 2 \sum_{j=1}^{k-1}m(j,k) - \frac{(k-1)^2}{2},
\end{align}
where
\begin{equation}\label{m}
m(j,k) = \#\biggl\{h: 1\leq h \leq \dfrac{k-1}{2};\left[\dfrac{(2h+k)j}{k}\right]\text{ is odd}\biggr\},
\end{equation}
and where in the last line of \eqref{4.10}, we added and subtracted the number of odd values of $\left[\frac{2hj+jk}{k}\right]$.

 Using \eqref{squarebrackets}, we have the representation,
\begin{align*}
m(j,k)=&\; \sum_{h=1}^{\frac{k-1}{2}}\left(\left[\frac{(2h+k)j}{k}\right] - 2\left[\frac{(2h+k)j}{2k}\right]\right)\notag\\
=&\; \sum_{h=1}^{\frac{k-1}{2}}\left(j+\left[\frac{2hj}{k}\right] - 2\left[\frac{hj}{k}+\frac{j}{2}\right]\right).
\end{align*}
\allowdisplaybreaks
Employing \eqref{elementary evaluation} in the fourth equality below, we find that
\begin{align}
\sum_{j=1}^{k-1} m(j,k)=&\; \sum_{\ell=1}^{\frac{k-1}{2}}\sum_{h=1}^{\frac{k-1}{2}} \Biggl(\left(2\ell + \left[\frac{4h\ell}{k}\right]- 2\left[\frac{2h\ell}{k}+\frac{2\ell}{2}\right]\right)\notag\\
&+ \left(2\ell -1 + \left[\frac{2h(2\ell-1)}{k}\right]- 2\left[\frac{h(2\ell-1)}{k}+\frac{2\ell-1}{2}\right]\right)\Biggr)\notag\\
=&\; \sum_{\ell=1}^{\frac{k-1}{2}}\sum_{h=1}^{\frac{k-1}{2}} \left(\left[\frac{4h\ell}{k}\right]+\left[\frac{2h(2\ell-1)}{k}\right]\right)\notag \\&- \sum_{\ell=1}^{\frac{k-1}{2}}\sum_{h=1}^{\frac{k-1}{2}}\left(2\left[\frac{2h\ell}{k}\right] + 2\left[\frac{h(2\ell-1)}{k}-\frac{1}{2}\right]+1\right)\notag\\
=&\; \sum_{\ell=1}^{k-1}\sum_{h=1}^{\frac{k-1}{2}}\left[\frac{2h\ell}{k}\right] -   2\sum_{\ell=1}^{\frac{k-1}{2}}\sum_{h=1}^{\frac{k-1}{2}}\left(\left[\frac{2h\ell}{k}\right] + \left[\frac{h(2\ell-1)}{k}-\frac{1}{2}\right]\right) - \frac{(k-1)^2}{4}\notag\\
=&\; \sum_{h=1}^{\frac{k-1}{2}}\frac{(2h-1)(k-1)}{2} -   2\sum_{\ell=1}^{\frac{k-1}{2}}\sum_{h=1}^{\frac{k-1}{2}}\left(\left[\frac{2h\ell}{k}\right] + \left[\frac{h(2\ell-1)}{k}-\frac{1}{2}\right]\right) - \frac{(k-1)^2}{4}\notag\\
=&\; \frac{k-1}{2}\sum_{h=1}^{\frac{k-1}{2}} (2h-1) - 2\sum_{\ell=1}^{\frac{k-1}{2}}\sum_{h=1}^{\frac{k-1}{2}}\left(\left[\frac{2h\ell}{k}\right] + \left[\frac{h(2\ell-1)}{k}-\frac{1}{2}\right]\right) - \frac{(k-1)^2}{4}\notag\\
=&\;\frac{k-1}{2}\cdot \frac{(k-1)^2}{4}- \frac{(k-1)^2}{4}- 2\sum_{\ell=1}^{\frac{k-1}{2}}\sum_{h=1}^{\frac{k-1}{2}}\left(\left[\frac{2h\ell}{k}\right] + \left[\frac{h(2\ell-1)}{k}-\frac{1}{2}\right]\right)\notag\\
=&\; \frac{(k-1)^2(k-3)}{8}- 2\sum_{\ell=1}^{\frac{k-1}{2}}\sum_{h=1}^{\frac{k-1}{2}}\left(\left[\frac{2h\ell}{k}\right] + \left[\frac{h(2\ell-1)}{k}-\frac{1}{2}\right]\right).
\label{mm}
\end{align}
It seems reasonable to  restate this last identity in terms of fractional parts, rather than greatest integer functions.  We therefore find that
\begin{gather}
\sum_{\ell=1}^{\frac{k-1}{2}}\sum_{h=1}^{\frac{k-1}{2}}\left(\left[\frac{2h\ell}{k}\right] + \left[\frac{h(2\ell-1)}{k}-\frac{1}{2}\right]\right)\notag\\ = \frac{(k-1)^3}{16}\,-\, \sum_{\ell=1}^{\frac{k-1}{2}}\sum_{h=1}^{\frac{k-1}{2}}\left(\left\{\frac{2h\ell}{k}\right\} + \left\{\frac{h(2\ell-1)}{k}-\frac{1}{2}\right\}\right).\label{mmm}
\end{gather}

 Finally, using \eqref{4.10}, \eqref{mm}, and \eqref{mmm}, we deduce that
\begin{align}\label{meyersum1}
S(k)  =&\; \frac{(k-1)^2(k-5)}{4} - 4\Biggl(\frac{(k-1)^3}{16}\,-\, \sum_{\ell=1}^{\frac{k-1}{2}}\sum_{h=1}^{\frac{k-1}{2}}\left(\left\{\frac{2h\ell}{k}\right\} + \left\{\frac{h(2\ell-1)}{k}-\frac{1}{2}\right\}\right)\Biggr)\nonumber\\
=&\; \frac{(k-1)^2(k-5)}{4} - \frac{(k-1)^3}{4}+4 \sum_{\ell=1}^{\frac{k-1}{2}}\sum_{h=1}^{\frac{k-1}{2}}\left(\left\{\frac{2h\ell}{k}\right\} + \left\{\frac{h(2\ell-1)}{k}-\frac{1}{2}\right\}\right)\nonumber\\
=&\; -(k-1)^2 + 4 \sum_{\ell=1}^{\frac{k-1}{2}}\sum_{h=1}^{\frac{k-1}{2}}\left(\left\{\frac{2h\ell}{k}\right\} + \left\{\frac{h(2\ell-1)}{k}-\frac{1}{2}\right\}\right),
\end{align}
and this completes the proof.
\end{proof}

  It is here that the difficulty of proving that $S(k)>0$ becomes apparent. If one surmises that the fractional parts on average are equal to $\tfrac{1}{2}$, then, on average, the right side of \eqref{meyersum1} should be close to 0.  However, the data do not show this.  So there must be a skewed distribution of the fractional parts so that the expressions in the sum are somehow larger than 1 on average.

We may pair the fractional parts in the sum of \eqref{meyersum1} as follows to form integers. Define
 $$f(l,h):=\left\{\frac{2hl}{k}\right\}+\left\{\frac{h(2l-1)}{k}-\frac{1}{2}\right\}.$$
\begin{proposition}\label{proposition6.2}
	For $m<\frac{k+1}{4}$, let
$$g(m,h):=f(m,h)+f\left(\frac{k-1}{2}-m+1,h\right).$$
 Then $g(m,h)$ is an integer.
\end{proposition}

\begin{proof}
	Rewrite $g(m,h)$ as
 $$g(m,h)=\left\{\frac{2mh}{k}\right\}+\left\{\frac{h}{k}-\frac{2mh}{k}\right\}+\left\{\frac{2mh}{k}-\frac{h}{k}-\frac{1}{2}\right\}+\left\{\frac{-2mh}{k}-\frac{1}{2}\right\}.$$
	We separate several cases. \\
\medskip

\noindent	\textbf{Case A}. $\left\{\frac{h}{k}\right\}\leq \left\{\frac{2mh}{k}\right\}$.  Then,
 \begin{align*}
		g(m,h)=&\left\{\frac{2mh}{k}\right\}+1+\left\{\frac{h}{k}\right\}-\left\{\frac{2mh}{k}\right\}+\left\{\frac{2mh}{k}-\frac{h}{k}-\frac{1}{2}\right\}
+1-\left\{\frac{2mh}{k}+\frac{1}{2}\right\}\\\nonumber
		=&1+\left\{\frac{h}{k}\right\}+\left\{\frac{2mh}{k}-\frac{h}{k}-\frac{1}{2}\right\}+1-\left\{\frac{2mh}{k}+\frac{1}{2}\right\}.
	\end{align*}
	We have the following sub-cases.

\medskip
	
	\textbf{Sub-Case A1}. If $\left\{\frac{2mh}{k}\right\}-\left\{\frac{h}{k}\right\}\geq \frac{1}{2}$, then we must have $\left\{\frac{2mh}{k}\right\}>\frac{1}{2}$, and so
	$$g(m,h)=1+\left\{\frac{h}{k}\right\}+\left\{\frac{2mh}{k}\right\}-\left\{\frac{h}{k}\right\}-\frac{1}{2}+1-\left(\left\{\frac{2mh}{k}\right\}+\frac{1}{2}-1\right)=2.$$
	
	\textbf{Sub-Case A2}. If $\left\{\frac{2mh}{k}\right\}-\left\{\frac{h}{k}\right\}< \frac{1}{2}$ and $\left\{\frac{2mh}{k}\right\}>\frac{1}{2}$, then
	$$g(m,h)=1+\left\{\frac{h}{k}\right\}+1+\left\{\frac{2mh}{k}\right\}-\left\{\frac{h}{k}\right\}-\frac{1}{2}+1-\left(\left\{\frac{2mh}{k}\right\}+\frac{1}{2}-1\right)=3.$$
	
	\textbf{Sub-Case A3}. If $\left\{\frac{2mh}{k}\right\}-\left\{\frac{h}{k}\right\}< \frac{1}{2}$ and $\left\{\frac{2mh}{k}\right\}<\frac{1}{2}$, then
	$$g(m,h)=1+\left\{\frac{h}{k}\right\}+1+\left\{\frac{2mh}{k}\right\}-\left\{\frac{h}{k}\right\}-\frac{1}{2}+1-\left(\left\{\frac{2mh}{k}\right\}+\frac{1}{2}\right)=2.$$

\medskip\noindent
	 \textbf{Case B}. $\left\{\frac{h}{k}\right\}>\left\{\frac{2mh}{k}\right\}$. Then, since $\left\{\frac{h}{k}\right\}<\frac{1}{2}$, we must have $\left\{\frac{2mh}{k}\right\}<\frac{1}{2}$. Thus
 $$g(m,h)= \left\{\frac{h}{k}\right\}+1+\left\{\frac{2mh}{k}\right\}-\left\{\frac{h}{k}\right\}-\frac{1}{2}+1-\left(\left\{\frac{2mh}{k}\right\}+\frac{1}{2}\right)=1.$$
 Since we have examined all possible cases, Proposition \ref{proposition6.2} follows.
\end{proof}

When $\frac{k-1}{2}$ is odd, i.e., $k\equiv3\pmod 4$, there is an unpaired term $f(l,h)$  in the sum \eqref{meyersum1} for $l=\frac{k+1}{4}$. Thus,
\begin{align}
    f(l,h)&=\left\{\frac{2h(k+1)}{4k}\right\}+\left\{\frac{h}{k}\frac{k-1}{2}-\frac{1}{2}\right\}\notag\\
    &=\left\{\frac{h}{2}+\frac{h}{2k}\right\}+\left\{\frac{h}{2}-\frac{h}{2k} -\frac{1}{2}\right\}\label{hk}
    \end{align}
Now, $\left\{\frac{h}{2}\right\}=\frac{1}{2}$ or $0$ and  $\left\{\frac{h}{2k}\right\}<\frac{1}{2}$.
Hence, from \eqref{hk},
\begin{align}
f(l,h)&=\left\{\frac{h}{2}+\frac{h}{2k}\right\}+\left\{\frac{h}{2}-\frac{h}{2k} -\frac{1}{2}\right\}\notag
\\&=\left\{\frac{h}{2}\right\}+\left\{\frac{h}{2k}\right\}+1+\left\{\frac{h}{2}\right\}-\left\{\frac{h}{2k}\right\}-\frac{1}{2}=2\left\{\frac{h}{2}\right\}+\frac{1}{2}.
\label{hk1}
\end{align}

\bigskip

If one can show that the number of pairs $(m,h)$ that satisfy condition A2 is greater than the number of pairs that satisfy condition B, then we have a proof of the conjecture $S(k)>0$ for prime $k$.

We establish another lower bound for $S(k)$ by employing Riemann--Stieltjes  integrals. The new bound is somewhat sharper than the earlier lower bounds in this section.  Let
$$ H(k):=\sum_{\substack{j=1\\j \text{ odd}}}^{k-1}\df{1}{j}.$$

\begin{theorem}\label{theorem2} If $k$ is an odd prime, then
 \begin{align*}
	S(k)
	&\geq -\frac{k-1}{2}+kH(k)-\frac{(k-1)(k+1)}{4},
\end{align*}
where, as $k\to\infty$,
 $$H(k)=\frac{1}{2}\log\left(2k\right)+\frac{\gamma}{2}+O\left(\frac{1}{k}\right),$$
where $\gamma$ denotes Euler's constant.
\end{theorem}

\begin{proof} Recall from \eqref{S_k_expression} the identity
\begin{align}\label{4.22}
	S(k)
	&=\frac{(k-1)^2}{2}-2\sum_{\substack{j=1\\j \text{ odd}}}^{k-1}r(j,k),
\end{align}
where  $r(j,k)$ is defined in \eqref{r}. From \eqref{m_jk_expression},
    \begin{align}\label{mmmm}
r(j,k)= \sum_{h=1}^{k-1} \left(\left[\frac{hj}{k}\right]-2\left[\frac{hj}{2k}\right]\right).
\end{align}
Thus, by \eqref{mmmm},
\begin{align}
 \lim_{k\to\infty}\frac{1}{k} r(j,k)&=\lim_{k\to\infty} \frac{1}{k}\sum_{h=1}^{k} \left(\left[\frac{hj}{k}\right]-2\left[\frac{hj}{2k}\right]\right)\nonumber\\
    &=\int _0^1\left(\left[ jx\right]-2\left[\frac{jx}{2}\right]\right)dx\nonumber\\
    &=\df{1}{j}\left(\frac{j-1}{2}\right).\label{riemannsum}
\end{align}
To obtain the last equality in \eqref{riemannsum}, first note that
\begin{equation}\label{4.111}
\left[ jx\right]-2\left[\frac{jx}{2}\right]
=\begin{cases} 1,\quad &\text{ if } k+\frac{1}{2} \leq \frac{jx}{2} \leq k+1,  \\
0,&\text{ otherwise}.
\end{cases}
\end{equation}
 Next, subdivide the interval $[0,1] $ into subintervals of length $\frac{1}{j}$. From \eqref{4.111}, we see that the largest value of $k$ is $\frac{j-1}{2}$. The value of the integral is therefore $\frac{1}{j}\cdot\frac{j-1}{2}$, and consequently \eqref{riemannsum} follows.

From the limit in \eqref{riemannsum}, we are motivated to define an error term $\epsilon_j(k)$ by
\begin{equation}\label{4.12}
\epsilon_j(k):= \frac{1}{k} r(j,k) -\frac{1}{j}\left(\frac{j-1}{2}\right).
\end{equation}
 We derive an upper bound for $\epsilon_j(k).$  Define
 \begin{equation}\label{4.24}
  f(j,k):=[ jx]-2\left[\frac{jx}{2}\right],
  \end{equation}
  as in \eqref{4.111}.  Subdivide the interval $[0,1]$ into subintervals of length $\frac{1}{j}.$ Consider the sequence of fractions
 $$\frac{1}{k},\frac{2}{k},\dots,\frac{k-1}{k},$$
 and set
 $$a_n=\frac{n}{k},\quad 1\leq n \leq k-1.$$
 Consider a point
 $$\frac{n}{j}, \quad 1\leq n \leq j-1.$$
   Let $a_m$ be the closest fraction to $\frac{n}{j}$ to the left. Of course, $a_{m+1}$ is the closest fraction to $\frac{n}{j}$ on the right.  We note that $|a_{m+1}-a_m|=\frac{1}{k}$.  From \eqref{4.111} and \eqref{4.24}, we see that the contribution of each of these subintervals to $f(j,k)$ is $\frac12\cdot\frac{1}{k}$. Since there are $j-1$ of these contributions, we conclude that
 $$ \epsilon_j(k)\leq \dfrac{j-1}{2k}.$$
 Therefore, by \eqref{4.12},
 \begin{equation}\label{4.14}
 r(j,k)\leq \frac{k}{2}-\frac{k}{2j}+\frac{j-1}{2} .
\end{equation}

Finally, from \eqref{4.22} and \eqref{4.14}, we conclude that
\begin{align}\label{4.15}
	S(k)
	&\geq \frac{(k-1)^2}{2}-2\sum_{\substack{j=1\\j\text{ odd}}}^{k-1}\left(\frac{k}{2}-\frac{k}{2j}+\frac{j-1}{2}\right) \nonumber\\
 &=-\frac{k-1}{2}+k\sum_{\substack{j=1 \\j \text{ odd}}}^{k-1}\frac{1}{j}-\frac{(k-1)(k+1)}{4}.
\end{align}
Recall the familiar asymptotic formula for the harmonic sum, as $k\to\infty$,
$$ \sum_{j=1}^k\df{1}{j} = \log k +\gamma +O\left(\df{1}{k}\right), $$
where $\gamma$ denotes Euler's constant.  Hence, we deduce that, as $k\to\infty$,
\begin{align}\label{4.16}
\sum_{\substack{j=1\\j\text{ odd}}}^{k-1}\df{1}{j}= \sum_{j=1}^{k-1}\df{1}{j}- \df{1}{2}\sum_{j=1}^{(k-1)/2}\df{1}{j}=&\;\frac{1}{2}\log\left(2(k-1)\right)+\frac{\gamma}{2}+O\left(\frac{1}{k-1}\right)\notag\\
=&\;\frac{1}{2}\log(2k)+\frac{\gamma}{2}+O\left(\frac{1}{k}\right).
\end{align}
Together, \eqref{4.15} and \eqref{4.16} complete the proof of Theorem \ref{theorem2}.
\end{proof}

This is a slight improvement over our previous lower bounds, but it is still not positive. As we can see from the proof above, a proper upper bound for $\epsilon_j(k)$ is crucial to gaining a better lower bound for $S(k)$, but to get a smaller upper bound for $\epsilon_j(k)$, we need to obtain a better understanding of the distribution of fractional parts in \eqref{meyersum1}, or the distribution of the residue classes of $hj$ modulo $2k$. In other words, we need to have a suitable estimate for the sum
$$\sum_{\substack{j=1\\j \textup{ odd }}}^{k-1}\sum_{h=1}^{k-1}\left\{\df{hj}{2k}\right\}.$$
 Observe that, since $j$ is odd, $\gcd(j,2k)=1$, and therefore $\{hj\}$ exhausts all of the nonzero residue classes modulo $2k$.  However, our range of $h$ is only $1\leq h\leq k-1$.

\section{Elementary results}\label{elementary}

As indicated in the Introduction, since Dedekind sums have been extensively studied in the literature, we are motivated to develop a corresponding theory for $S(h,k)$ and the concomitant sums examined in this paper.  See \cite{jm1} and \cite{jm2} for earlier elementary results on $S(h,k)$. Our proofs are elementary.

\begin{proposition}\label{distrib}
    For primes $k$, the values of $S(k)$ are equally distributed modulo $4$.
\end{proposition}

More precisely, there are only two possible cases for primes $k$, namely  $S(k)\equiv 0 \pmod{4}$ and $S(k)\equiv 2 \pmod{4}$, and these two residue classes are evenly distributed.

\begin{proof} By considering the cases $k\equiv 1,3 \pmod 4$ separately, we can easily establish Proposition \ref{distrib}. \end{proof}

\begin{proposition}\label{parity} For each positive integer $k$,
    $S(k)$ and $k$ have opposite parity.
    \end{proposition}

\begin{proof}
The desired result is immediate upon considering the parities of $S(h,k)$ and $k$.


\end{proof}

\begin{corollary} The only pair for which both $k$ and $S(k)$ are prime is $(3,2)$.
\end{corollary}

\begin{proposition} When $h=k-1$, $S(h,k) = h$.
 \end{proposition}

\begin{proof}
If $h = k-1$, then for each value of $j, 1\leq j \leq k$, we see that $$\left[\frac{hj}{k}\right] = \left[\frac{(k-1)  j}{k}\right] = j-1.$$  Hence,
$$S(k-1,k)=\sum_{j=1}^{k-1}1=k-1.$$
\end{proof}

\begin{proposition}\label{4.7} If $h$ and $k$ are odd, co-prime, positive integers, then
\begin{equation*}
S(h,k)=0.
\end{equation*}
\end{proposition}

\begin{proof}
For $1\leq j \leq \frac{k-1}{2}$, by an elementary calculation, we see that
\begin{align}\label{4.8}
(-1)^{k-j+1+[(k-j)h/k]}
=-(-1)^{j+1+[hj/k]}.
\end{align}
It follows from \eqref{defSkh} and \eqref{4.8} that $S(h,k)=0$.
\end{proof}

The following proposition relates $S(h,k)$ and $T(h,k)$.  Thus, if we prove a result about $T(h,k)$, we can obtain a corresponding result for $S(h,k)$.

\begin{proposition}\label{hkhk}  If $h$ and $k$ are positive, co-prime integers, then
\begin{equation*}
T(h,k) = \bigl(1+(-1)^{k+h}\bigr) S(h,k) + (-1)^{k+h+1}.
\end{equation*}
\end{proposition}

\begin{proof}  We have
\allowdisplaybreaks
\begin{align*}
T(h,k) =&\; \sum_{j=1}^{2k-1} (-1)^{j+1 + [hj/k]}\\
=&\; \sum_{j=1}^{k-1} (-1)^{j+1 + [hj/k]} + (-1)^{k+1+[hk/k]} + \sum_{j=k+1}^{2k-1} (-1)^{j+1 + [hj/k]}\\
=&\; S(h,k) + (-1)^{k+h+1} + \sum_{j=1}^{k-1} (-1)^{k+j+1 + [h(k+j)/k]}\\
=&\; S(h,k) + (-1)^{k+h+1} + (-1)^{k+h} \sum_{j=1}^{k-1} (-1)^{j+1 + [hj/k]}\\
=&\; \bigl(1+(-1)^{k+h}\bigr) S(h,k) + (-1)^{k+h+1}.
\end{align*}
\end{proof}

\begin{corollary}
    \label{g1} If $h$ and $k$ are positive, co-prime integers of opposite parity, then
$$T(h,k) = 1.$$
\end{corollary}

\begin{proof} The corollary follows immediately from Propositions \ref{4.7} and \ref{hkhk}.
\end{proof}

\begin{theorem} Let $\gcd(h,k)=1$ with $k>0$, and let $q$ be any positive integer.  Then
\begin{equation*}
S(qh,qk) = \begin{cases} S(h,k), &\text{if $h+k$ and $q$ are odd};\\
1,  &\text{if $h+k$ is odd and $q$ is even};\\
-(q-1),  &\text{if $h$ and $k$ are odd}.\\
\end{cases}
\end{equation*}
\end{theorem}

\allowdisplaybreaks
\begin{proof}  From the definition \eqref{defSkh}, we have
\begin{align*}
S(qh, qk) =&\; \sum_{j=1}^{qk-1}(-1)^{j+1+[hj/k]}\\
=&\; \sum_{\ell=0}^{q-1}\sum_{j=\ell k +1}^{(\ell+1)k -1}(-1)^{j+1+[hj/k]} + \sum_{\ell=1}^{q-1} (-1)^{\ell k + 1 + \ell h}\\
=&\; \sum_{\ell=0}^{q-1}\sum_{j=1}^{k -1}(-1)^{j+\ell k+1+[h(j+\ell k)/k]} - \sum_{\ell=1}^{q-1} (-1)^{\ell (h + k)}\\
=&\; \sum_{\ell=0}^{q-1}(-1)^{\ell(h+k)}\sum_{j=1}^{k -1}(-1)^{j+1+[hj/k]} - \sum_{\ell=1}^{q-1} (-1)^{\ell (h + k)}\\
=&\; S(h,k)\sum_{\ell=0}^{q-1}(-1)^{\ell(h+k)} - \sum_{\ell=1}^{q-1} (-1)^{\ell (h + k)}.
\end{align*}
The theorem now follows upon an analysis of the sums on $\ell$ for the different parities of $h,k$, and $q$.
\end{proof}

As consequences of Lemma \ref{classify}, we obtain the following results.

\begin{corollary} For each prime $p$,
    $$T(p) = 1-2p.$$
\end{corollary}

\begin{proof} From Lemma \ref{classify},
 $$T(p)= 2p-1 - 2\sum\limits_{\substack{1\leq h \leq 2p - 1 \\ h \text{ odd}}}  \gcd(h,p).$$
Since $p$ is prime, $\gcd(h,p) = 1$ for all $h$ that are not multiples of $p$. In the interval of interest, there are $p$ odd numbers, and the only multiple of $p$ is $p$ itself. Thus,
$$
T(p) = 2p-1-2\cdot(\gcd(p,p)) - 2\cdot(p-1) = 2p-1-2p-2p+2 = 1-2p.
$$
\end{proof}

\begin{corollary} If $p$ and $q$ are odd primes,
    $$T(pq) = 4(p+q) - 6pq - 3.$$
\end{corollary}
\begin{proof} From Lemmas \ref{classify} and \ref{classify2}, we have
\begin{equation}\label{classify7}
T(pq)= 2pq-1 - 2\sum\limits_{\substack{1\leq h \leq 2pq- 1 \\ h \text{ odd}}}  \gcd(h,pq).
\end{equation}
First, consider the interval $[1,pq]$.  Other than $pq$, there are $(q-1)$ multiples of $p$, and there are $(p-1)$ multiples of $q$.  Thus, in $[1,pq]$, this contribution to the sum on the right side is
\begin{equation}\label{classify5}
-2\bigl((q-1)p+(p-1)q+1\bigr).
\end{equation}
Because $\gcd(j,pq)=\gcd(pq-j,pq), 1\leq j\leq pq-1$, this contribution of the terms in $[pq+1,2pq-1]$ is, by \eqref{classify5},
\begin{equation}\label{classify6}
-2\bigl((q-1)p+(p-1)q\bigr).
\end{equation}
In conclusion, from \eqref{classify7}--\eqref{classify6},
\begin{equation*}
T(pq)=2pq-1-4\bigl((q-1)p+(p-1)q\bigr)-2=-6pq+4p+4q-3.
\end{equation*}
\end{proof}

We provide a list of elementary results on $T(k)$ when $k$ is even.  The proofs of the analogue of Lemma \ref{classify} along with the necessary lemmas are similar to those above.

\begin{lemma}\label{even} When $k$ is even,
$$T(k)=
2k-1 - \sum\limits_{\substack{1\leq h \leq 2k - 1 \\ h \equiv 2 \pmod{4}}}  2\gcd(h,k).$$
\end{lemma}

\begin{proposition} For $k$ even and $h$ odd,
$$T(h,k) = 1. $$
\end{proposition}

\begin{proposition} If $k \equiv 2 \pmod 4$ and $h \equiv 0 \pmod 4$, then
$$T(h,k) = 1. $$
\end{proposition}

\begin{proposition}
If $k \equiv 2 \pmod 4$ and $h \equiv 2 \pmod 4$, then
$$T( h,k) = 1 - 2 \gcd(h,k).$$
\end{proposition}

\begin{proposition}
If $k \equiv 0 \pmod 4$ and $h\equiv 0\pmod 4$, then 
$$T( h,k) = 1 - 2 \gcd(h,k).$$
\end{proposition}

\begin{corollary}
    For all natural numbers $n$,
$$T\left(2^n\right) = -1.$$
\end{corollary}

\section{Acknowledgements}  We are grateful to Wolfgang Berndt for his comments and calculations at the conclusion of Section \ref{conjectures}, and to the referee for very helpful suggestions.

\section{Financial Declaration}  The authors declare that they do not have any involvement with any organization 
with financial interest, or with non-financial interest, in the subject matter discussed in this manuscript.

\end{document}